\documentclass[11pt,a4paper]{article}

\usepackage{amscd}
\usepackage{enumerate}

\usepackage{geometry}

\usepackage{amsmath, amssymb, latexsym}
\usepackage{bbm}
\usepackage{authblk}
\usepackage{amsfonts}
\usepackage{amsthm}
\usepackage{relsize}
\usepackage{setspace}
\usepackage{geometry}
\usepackage{url}
\usepackage{xspace}
\usepackage{tocloft}
\usepackage{graphics}
\usepackage{graphicx}
\usepackage{lscape}
\usepackage{microtype}
\usepackage[T1]{fontenc}
\usepackage{ulem}
\usepackage{mathrsfs}

\usepackage[usenames, dvipsnames]{color}
\usepackage[utf8]{inputenc}
\usepackage{tikz}

\usepackage[pagebackref=true]{hyperref}
\usepackage[alphabetic]{amsrefs}
\usepackage[english]{babel}

\newtheorem{theorem}{Theorem}[section]
\newtheorem{proposition}[theorem]{Proposition}

\newtheorem{corollary}[theorem]{Corollary}
\newtheorem{lemma}[theorem]{Lemma}

\theoremstyle{definition}
\newtheorem{definition}[theorem]{Definition}

\newtheorem{remark}[theorem]{Remark}
\newtheorem{notation}[theorem]{Notation}

\theoremstyle{problem}

\usepackage[alphabetic]{amsrefs}

\newcommand{\Aut}{\mathrm{Aut}}

\newcommand{\proj}{\mathrm{proj}}
\newcommand{\RR}{\mathbf{R}}
\newcommand{\ZZ}{\mathbf{Z}}

\newcommand{\NN}{\mathbf{N}}

\newcommand{\Min}{\mathrm{Min}}

\newcommand{\bd}{\partial}

\newcommand{\id}{\operatorname{id}}

\newcommand{\T}{\operatorname{\mathcal{T}}}
\newcommand{\G}{\operatorname{\mathbb{G}}}
\newcommand{\dist}{\operatorname{d}}

\newcommand{\Sym}{\operatorname{Sym}}

\newcommand{\sign}{\operatorname{sign}}
\newcommand{\spn}{\operatorname{span}}

\title{Parabolically induced unitary representations \\of the universal group $U(F)^+$ are $C_0$}

\author[1]{Corina Ciobotaru\thanks{corina.ciobotaru@gmail.com}}
\affil[1]{Universit\'e de Fribourg, Section de Math\'ematiques, Chemin du Mus\'{e}e 23, 1700 Fribourg, Switzerland.}

\date{November 21, 2018}

\begin{document}

\newcounter{qcounter}

\maketitle

\begin{abstract}
We prove that all parabolically induced unitary representations of the Burger--Mozes universal group $U(F)^{+}$, with $F$ being primitive, are $C_0$. This generalizes the same well-known result for the universal group $U(F)^{+}$, when $F$ is $2$--transitive.
\end{abstract}

\section{Introduction}
Let $\mathbf{\T}$ be a $d$-regular tree, with $d \geq 3$. Let $\G:= U(F)^+ \leq \Aut(\T)$, with $F \leq \Sym\{1,\cdots, d\}$ being primitive, be the universal group introduced by Burger and Mozes in~\cite[Sec.~3]{BM00a}. Given a (strongly continuous) unitary representation  $\pi :\G \to \mathcal{U}(\mathcal{H}) $ on a (infinite dimensional) complex Hilbert space $(\mathcal{H}, \left\langle \cdot , \cdot \right\rangle)$, we are interested in studying the matrix coefficients $c_{v,w} : \G \to \mathbb{C}$  given by $c_{v,w}(g):=\left\langle \pi(g)v ,w \right\rangle$, for every $v,w \in \mathcal{H}$. We say $(\pi, \mathcal{H})$ is a $C_0$ unitary representation of $\G$ if for any of its associated matrix coefficients $c_{v,w}$, the subset $\{g \in \G \; \vert \; \vert c_{v,w}(g)\vert \geq \epsilon \}$ is compact in $\G$, for every $\epsilon >0$; equivalently,  $\lim\limits_{g \to \infty} \vert c_{v,w}(g) \vert =0$, for every $v, w \in \mathcal{H}$, where $\infty$ represents the one-point compactification of the locally compact group $\G$. 

It is a general fact \cite[Appen.~C, Prop.~C.4.6]{BHV} the left regular unitary representation of $\G$ is $C_0$. When  $F$ is $2$--transitive (if and only if $\G$ is $2$--transitive on the boundary $\partial \T$) \cite{BM00a} showed the Howe--Moore property of $\G$: every unitary representation of $\G$, without non-zero $\G$-invariant vectors, is $C_0$. Still, for $F$ being just primitive, but not $2$--transitive, it is difficult to predict when a (non-trivial) unitary representation of $\G$ is $C_0$ or not. Apart from \cite[Prop.~C.4.6]{BHV} and the general criterion proven in~\cite{Cio} (and the references therein) providing a unified proof of the Howe--Moore property for all known examples, there are no other known techniques to prove a unitary representation of a locally compact group is $C_0$.

\cite{Cio_rel}) shows  $\G$, when $F$ is primitive, has a weakening of the Howe--Moore property, namely the relative Howe--Moore property with respect to any \textbf{horospherical stabilizer} $\G_{\xi}^{0}:=\{g \in  \G \; \vert \; g(\xi)=\xi, \; g \text{ elliptic } \}$ with $\xi \in \bd \T$. This relative property was introduced and studied in~\cite{CCL+}. Another result of~\cite{Cio_rel} shows when $F$ is primitive but not $2$-transitive, the stabilizer $\G_v$ in $\G$ of a non-zero vector $v \in \mathcal{H}$ of a unitary representation $(\pi,\mathcal{H})$ of $\G$ without non-zero $\G$-invariant vectors, either is compact or if it is not compact then it equals $\G_{\xi}^{0}$, for some $\xi \in \bd \T$.  It is then natural to ask whether the unitary representations of $\G$ induced from closed subgroups of the stabilizer $\G_\xi:=\{g \in  \G \; \vert \; g(\xi)=\xi \}$ for $\xi \in \bd \T$ are $C_0$ or not.  The following vanishing result gives the answer to this question. For  the theory of  induced unitary representations we use notation and the results from~\cite[Appen.~B,~E]{BHV}.

\begin{theorem}
\label{mainthm::vanishing_general_case}
Let $F \leq \Sym\{1,\cdots, d\}$ be primitive and $\xi \in \bd \T$. Let $H$ be a closed subgroup of $\G$ stabilizing $\xi$  and let $(\sigma, \mathcal{K})$ be a unitary representation of $H$. Then the induced unitary representation $(\pi_{\sigma}, \mathcal{H_{\sigma}})$ on $\G$ is $C_0$.
\end{theorem}

We emphasise  Theorem~\ref{mainthm::vanishing_general_case} covers the known case when $F$ is $2$--transitive (that case being covered by the Howe--Moore property). Still, the proof of Theorem~\ref{mainthm::vanishing_general_case}  is very different from the general one proving the Howe--Moore property. This is firstly, because the group $\G$, when $F$ is primitive but \textit{not} $2$--transitive, does not verify the general criterion given in~\cite{Cio}. Secondly, if $F$ is primitive but \textit{not} $2$--transitive, it is a direct consequence  the quotient $\G / \G_{\xi}$ is \textit{not} compact anymore and \textit{not} isomorphic to the boundary $\bd \T$, for any choice of $\xi \in \bd \T$. Moreover, by~\cite{Cio_rel}, $ \G_{\xi}$ is a closed, still non-compact subgroup, for every $\xi \in \bd \T$. 

To prove Theorem \ref{mainthm::vanishing_general_case} we follow the lines of the standard argument that the left regular unitary representation is $C_0$. The novelty of the article is the control of the integral given by \ref{rem::evaluate_rho_function} in terms of indices of subgroups and we distinguish three cases in the calculation of the asymptotics of that integral.

\section{Some properties of $\G$}
\label{sec::univ_group}

For the definition and the main properties of the Burger--Mozes universal groups the reader can consult Burger--Mozes~\cite{BM00a},  Amann~\cite{Amann}, Ciobotaru \cite{Cio_par}.

To fix the notation, let $\dist_{\T}(\cdot, \cdot)$ be the usual metric on $\T$. Let $\Aut(\T)^{+}$ be the group of all type-preserving automorphisms of $\T$ and by definition $\G \leq \Aut(\T)^{+}$.
 For every two points $x,y \in \T \cup \bd \T$,  $[x,y]$ is the unique geodesic between $x$ and $y$ in $\T \cup \bd \T$. For $G \leq \Aut(\T)$ and $x, y \in \T\cup \bd \T$ we define $G_{[x,y]}:=\{ g \in G \; \vert \; g \text{ fixes pointwise}$ $\text{the geodesic } [x,y]\}.$ In particular, $G_{x}:= \{g \in G \; \vert \; g(x)=x\}$. For $\xi \in \bd \T$ we have already defined $G_{\xi},G_{\xi}^{0}$. Note $G_{\xi}$ can contain hyperbolic elements; if this is the case then $G_{\xi}^{0} \lneq G_{\xi}$. If $H \leq G_{\xi}$ and $x \in \T$ then $H_x$ evidently equals $H_{[x, \xi]}$. For a vertex $x \in \T$ and an edge $e$ in the star of $x$, set $K:=G_{x}$ and let $\T_{x,e}$ be the half-tree of $\T$ emanating from the vertex $x$ and containing the edge $e$. For a hyperbolic element $\gamma \in \Aut(\T)$, we denote $\vert \gamma \vert:= \min_{x \in \T}\{ \dist_{\T}(x, \gamma(x))\}$, which is called the translation length of $\gamma$. Set $\Min(\gamma)=\{x \in \T \; \vert \;  \dist_{\T}(x, \gamma(x))$ $=\vert \gamma \vert\}$. 

\begin{remark}
\label{rem::existence_hyp_elements}
As $F$ is primitive, given an edge $e' \in E(\T) $ at odd distance from $e$, one can construct, using the definition of $\G$, a hyperbolic element in $\G$ translating $e$ to $e'$. Moreover, every hyperbolic element in $\G$ has even translation length, as $\G$ has only type-preserving automorphisms.
\end{remark}

\begin{lemma}($KA^{+}K$ decomposition)
\label{lem::KAK_decomposition}
Let $F$ be primitive. Let $x \in \T$ be a vertex and $e$ an edge in the star of $x$. Then $\G$ admits a $KA^{+}K$ decomposition, where $A^{+}:=\{\gamma \in \G \; \vert \;  e \subset \Min(\gamma), \gamma(e) \subset  \T_{x,e}  \} \cup \{\id\}.$
\end{lemma}

\begin{proof}
Let $g \in \G$. If $g(x)=x$, then $g \in K$. If not, consider the geodesic segment $[x, g(x)]$ in $\T$; denote by $e_1$ the edge of the star of $x$ belonging to $[x, g(x)]$. By type-preserving, $[x, g(x)]$ has even length. As $F$ is also transitive, there is $k \in K$ with $k(e_1)=e$; therefore, $kg(x) \in \T_{x,e}$. By~\ref{rem::existence_hyp_elements}, there is a hyperbolic element  $\gamma \in \G$ of translation length equal to the length of $[x, g(x)]$, translating the edge $e$ inside $\T_{x,e}$ and with $\gamma(x)=kg(x)$; thus $\gamma^{-1}kg \in K$. Note the $KA^{+}K$ decomposition of an element $g \in \G$ is not unique.
\end{proof}


\begin{lemma}
\label{lem::infinit_index}
Let $F$ be primitive and let $H$ be a closed, non-compact and proper subgroup of $\G$.  Then, for every $x \in \T$, $H_x$ does not have finite index in $\G_x$.
\end{lemma}

\begin{proof}
By Caprace--De Medts~\cite[Prop.~4.1]{CaMe11} the subgroup $F$ is primitive if and only if every proper open subgroup of $\G$ is compact.  $H$ cannot be an open subgroup of $\G$, as otherwise $H$ would be compact, contradicting the hypothesis. Suppose there is $x \in \T$ with $[\G_x: H_x] < \infty$. As $H_x$ is closed in $\G_x$ and of finite index, $H_x$ is open in $\G_x$ and so also in $\G$. This means $H$ is open in $\G$, obtaining a contradiction. 
\end{proof}

\section{Induced unitary representations}
\label{sec::induced_unitary_rep}
 

We follow~\cite[Appendices~B and~E]{BHV} where all the definitions, notation, proofs and complementary definitions can be found (see also~\cite{Cio_par}). Fix in this section $G$ to be a locally compact group and $H \leq G$ a closed subgroup. All Haar measures used in this paper are considered to be left invariant. We denote by $dx$, respectively, $dh$ the Haar measure on $G$, respectively, $H$. We endow $G/H$ with the quotient topology: the canonical projection $p :G \to G/H$ is continuous and open.  

\begin{definition}(See \cite[Appendix~B]{BHV})
\label{def::rho_function}
A \textbf{rho-function} of $(G,H)$ is a continuous function $\rho : G \to \RR^{*}_{+}$ satisfying the equality
\begin{equation}
\label{equ::rho_fct_cond}
\rho(xh)=\frac{\Delta_H(h)}{\Delta_G(h)} \rho(x) \text{ for all } x \in G, h \in H,
\end{equation}where $\Delta_{G}, \Delta_H$ are the modular functions on $G$, respectively on $H$. 
\end{definition}

By \cite[Thm.~B.1.4]{BHV}, there is a correspondence between rho-functions of $(G,H)$ and continuous $G$--quasi-invariant regular Borel (CGQIRB) measures on $G/H$ (see \cite[Appendix. A.3]{BHV}), where continuous means the Radon--Nikodym derivative of $\mu$ is continuous.
 
\begin{definition}
\label{def::induced_unit_rep}
Let $(\sigma, \mathcal{K})$ be a unitary representation of $H$. Suppose $G/H$ is endowed with a CGQIRB-measure $\mu$, with associated rho-function $\rho$ on $G$. The \textbf{induced unitary representation $(\pi_{\sigma, \mu}, \mathcal{H}_{\sigma, \mu})$} of $G$ is defined as follows. For every $g \in G$, we define the unitary operator $\pi_{\sigma, \mu}(g)$ on $\mathcal{A}$ by $\pi_{\sigma, \mu}(g) (\xi)(x):= \left(\frac{\rho(g^{-1}x)}{\rho(x)}\right)^{1/2} \xi(g^{-1}x),$ where $\xi \in \mathcal{A}$ and $x \in G$, and where $\mathcal{A}$ is a specific dense subset of the Hilbert space $\mathcal{H}_{\sigma, \mu}$. For a complete definition see~\cite[Appendix~E]{BHV}.  Moreover, by~\cite[Prop.~E.1.4]{BHV}, this is a unitary representation of $G$ on the Hilbert space $\mathcal{H}_{\sigma, \mu}$. 
\end{definition}


\begin{remark}
\label{rem::equiv_induced_unit_rep}
By~\cite[Proposition~E.1.5]{BHV} induced unitary representations do only depend on the unitary representations $(\sigma, \mathcal{K})$ of $H$ and not on the CGQIRB-measures on $G/H$. If $(\pi_{\sigma, \mu_{1}}, \mathcal{H}_{\sigma,\mu_{1}})$ is $C_0$ the same is true for $(\pi_{\sigma, \mu_{2}}, \mathcal{H}_{\sigma,\mu_{2}})$. 
\end{remark}

\begin{notation}
\label{notation_induced_rep}
By~\ref{rem::equiv_induced_unit_rep} it is legitimate to denote $(\pi_{\sigma}, \mathcal{H}_{\sigma})$ the unitary representation of $G$ induced from the unitary representation $(\sigma, \mathcal{K})$ of $H$. 
\end{notation}

\begin{lemma}
\label{lem::K-inv_borel_measure}
For every compact subgroup $K$ of $G$ there exists a CGQIRB-measure $\mu$ on $G/H$ which is left $K$--invariant.
 \end{lemma}

\begin{proof}
By~\cite[Thm.~B.1.4]{BHV} let $\mu_1$ be a CGQIRB-measure on $G/H$ with $\rho_1 : G \to \mathbb{R}^*_{+}$ its associated rho-function.  Let $\rho :G \rightarrow \mathbb{R}_{+}^{*}$ be the function defined by $g \in G \mapsto \rho(g):= \int\limits_{K} \rho_1(kg) dk$, with $dk$ the Haar measure on $K$. Then $\rho$ is continuous, satisfies equation~(\ref{equ::rho_fct_cond}) from \ref{def::rho_function} and so $\rho$ is a rho-function and left $K$--invariant. By~\cite[Thm.~B.1.4]{BHV}, let $\mu$ be the CGQIRB-measure on $G/H$ associated with $\rho$.  As the Radon--Nikodym derivative of $\mu$ satisfies $\frac{dy_{\star}\mu}{d\mu}(xH)=\frac{\rho(yx)}{\rho(x)}$, for every $x,y \in G$ and because $\rho$ is left $K$--invariant, we obtain $\mu$ is left $K$--invariant.
\end{proof}

\begin{lemma}
\label{lem::compact_open_quotient}
Let $K \leq G$ be compact. Consider on $G/H$ a CGQIRB-measure $\mu$ which is left $K$--invariant and suppose $\mu(KH) \neq 0$. If $K' < K$ is a compact subgroup of infinite index in $K$, with $H \cap K \leq K'$, then the index of $K'$ in $K$ is uncountable and $\mu(K'H)=0$. In particular, if $H\cap K$ has infinite index in $K$ then the index of $H\cap K$ in $K$ is uncountable and $\mu(H)=0$.
\end{lemma}

\begin{proof}
By~\ref{lem::K-inv_borel_measure}, we know $G/H$ admits a CGQIRB-measure $\mu$ which is left $K$--invariant. By the definition of a regular Borel measure $\mu(KH)< \infty$. Suppose the index of $K'$ in $K$ is countable; there exist $\{k_n\}_{n \in \mathbb{N}} \subset K \setminus K'$ with $K=\bigsqcup_{n \in \mathbb{N}} k_n K'$. Then $KH=\bigsqcup_{n \in \mathbb{N}} k_n K'H$. Indeed, if $k_n K'H \cap k_m K'H \neq \emptyset$ for some $n \neq m$ we would have $k_n k'= k_m k'' h$, for some $h \in H$ and some $k', k'' \in K'$; so $h \in H \cap K \leq K'$ and thus $k_n K'=k_m K'$, which is a contradiction. Therefore, write $\mu(KH)=\sum\limits_{n \in \mathbb{N}} \mu(k_nK'H)= \sum\limits_{n \in \mathbb{N}} \mu(K'H)$, as $\mu$ is countably additive and left $K$--invariant. Because $\mu(K'H), \mu(KH)< \infty$ we conclude $\mu(K'H)$ must be zero and so $\mu(KH)$ is zero too, which contradicts the hypothesis. Therefore, the index of $K'$ in $K$ must be uncountable. 
By the countable additivity of $\mu$, $K$--invariance of $\mu$ and $\mu(KH) \neq 0$, one easily obtains $\mu(K'H)=0$.
\end{proof}

\begin{remark}
\label{rem::radon_nykodym}
 By~\cite[Thm.~B.1.4]{BHV} let $\mu$ be a CGQIRB-measure on $G/H$ with associated rho-function $\rho$. Let $K$ be a compact subset of $G$. Then, for every $g \in G$, $g_{\star}\mu(KH)=\mu(g^{-1}KH) =\int\limits_{G/H}\mathbf{1}_{KH}(xH)dg_{\star}\mu(xH)=$
 
 $= \int\limits_{KH} \frac{\rho(gx)}{\rho(x)}d\mu(xH) \leq C_g \int\limits_{KH}\mathbf{1}_{KH}(x) d\mu(xH)= C_g \cdot \mu(KH)$
 
 where $C_g \geq \max_{x \in K}\{ \frac{\rho(gx)}{\rho(x)}\}$.  
\end{remark}

\begin{lemma}
\label{lem::first_reduction_1}
Let $(\sigma, \mathcal{K})$ be a unitary representation of $H$. Assume the induced unitary representation $(\pi_{\sigma}, \mathcal{H_{\sigma}})$ on $G$ is not $C_0$. Then there exist $\eta_1', \eta_2' \in \spn(\{\xi_{f,v} \; \vert \; f \in  C_{c}(G), v \in \mathcal{K} \})$, $\delta >0$ and a sequence $\{t_{k}\}_{k >0} \subset G$, with $t_{k} \to \infty$, such that $\vert \left\langle \pi_{\sigma}(t_{k})\eta_1', \eta_2' \right\rangle \vert >\delta$, for every $k>0$.
\end{lemma}

\begin{proof} Follows from $\spn(\{\xi_{f,v} \vert f \in  C_{c}(G), v \in \mathcal{K} \})$ is dense in $\mathcal{H_{\sigma}}$. 
\end{proof}

\begin{lemma}
\label{lem::first_reduction_2}
Let $K$ be an open-compact neighborhood in $G$ of the identity. Let $(\sigma, \mathcal{K})$ be a unitary representation of $H$ and $\eta_1, \eta_2 \in \spn(\{\xi_{f,v} \; \vert \; f \in  C_{c}(G), v \in \mathcal{K} \})$. Consider on $G/H$ a CGQIRB-measure $\mu$, with associated rho-function $\rho$ on $G$. 

Then there exist a constant $C>0$, $N_{1}, N_{2} \in \NN$ and elements $\{h_i\}_{i \in \{1,...,N_1\}}$, $\{h_j'\}_{j\in \{1,..., N_2\}} \subset G$, all of them depending only on $\eta_1$ and $\eta_2$, such that

$ \vert \left\langle \pi_{\sigma}(t)\eta_1, \eta_2 \right\rangle \vert=\vert \left\langle \eta_1, \pi_{\sigma}(t^{-1})\eta_2 \right\rangle \vert \leq$

$ \leq \sum_{i,j=1}^{N_{1},N_{2}}  \int\limits_{t(h_{i}KH) \cap h_{j}'KH} \left\vert \left(\frac{\rho(t^{-1}x)}{\rho(x)}\right)^{1/2} \left\langle \eta_1(t^{-1}x), \eta_2 (x )\right\rangle_{\mathcal{K}} d\mu (xH)\right\vert \leq$ 

$\leq  C \sum_{i,j=1}^{N_{1},N_{2}}  \int\limits_{t(h_{i}KH) \cap h_{j}'KH} \left(\frac{\rho(t^{-1}x)}{\rho(x)}\right)^{1/2} d\mu(xH)$

 for every $t \in G$. Moreover, we have 

$\int\limits_{t(h_{i}KH) \cap h_{j}'KH}  \left\vert \left(\frac{\rho(t^{-1}x)}{\rho(x)}\right)^{1/2} \left\langle \eta_1(t^{-1}x), \eta_2 (x )\right\rangle_{\mathcal{K}} d\mu(xH) \right\vert=$ 

$=\int\limits_{h_{i}KH \cap t^{-1}(h_{j}'KH)} \left\vert \left(\frac{\rho(ty)}{\rho(y)}\right)^{1/2} \left\langle \eta_1(y), \eta_2 (ty )\right\rangle_{\mathcal{K}} d\mu(yH) \right\vert.$
\end{lemma}

\begin{proof}
By~\ref{notation_induced_rep} we simply refer to $(\pi_{\sigma, \mu}, \mathcal{H_{\sigma, \mu}})$ as $(\pi_{\sigma}, \mathcal{H_{\sigma}})$.

Let $t \in G$. As $\eta_1, \eta_2 \in \spn(\{\xi_{f,v} \; \vert \; f \in  C_{c}(G), v \in \mathcal{K} \})$, they only depend on a finite number of functions from $C_{c}(G)$. Denote by $A, B \subset G$ the union of the support of those functions  defining $\eta_1$, respectively, $\eta_2$. $A$ and $B$ are compact subsets of $G$. Cover $A$, respectively, $B$, with open sets of the form $hK$, where $ h \in A$, respectively, $ h \in B$. From these open covers extract finite ones covering $A$, respectively, $B$. By making a choice and fixing the notation, consider $A \subset \bigcup\limits_{i=1}^{N_{1}}h_{i}K$ and $B \subset \bigcup\limits_{j=1}^{N_{2}}h_{j}'K$, where $h_i, h_{j}' \in G$ and $N_1, N_2 \in \NN$. We obtain:

\begin{equation}
\label{emu::first_reduction}
\begin{split}
&\vert \left\langle \pi_{\sigma}(t)\eta_1, \eta_2 \right\rangle \vert = \left\vert \text{ } \int\limits_{G/H} \left(\frac{\rho(t^{-1}x)}{\rho(x)}\right)^{1/2} \left\langle \eta_1(t^{-1}x), \eta_2 (x )\right\rangle_{\mathcal{K}} d\mu(xH)\right \vert \\
\end{split}
\end{equation}
\begin{equation*}
\begin{split}
& \leq \sum_{i,j=1}^{N_{1},N_{2}}  \int\limits_{t(h_{i}KH) \cap h_{j}'KH} \left\vert \left(\frac{\rho(t^{-1}x)}{\rho(x)}\right)^{1/2} \left\langle \eta_1(t^{-1}x), \eta_2 (x )\right\rangle_{\mathcal{K}}\right\vert d\mu(xH).\\
\end{split}
\end{equation*}

To obtain the last inequality of the lemma and the constant $C$, we use the following.
Recall $\eta_1, \eta_2 \in \spn(\{\xi_{f,v} \; \vert \; f \in  C_{c}(G), v \in \mathcal{K} \})$. We claim the scalar product $\vert \left\langle \eta_1(t^{-1}x), \eta_2 (x )\right\rangle_{\mathcal{K}} \vert$ is a bounded function in $x \in G$ and this upper-bound depends neither on $t$ nor on the domains $\{t(h_{i}KH) \cap h_{j}'KH\}_{h_i, h_j'}$. Indeed, for simplicity, consider $\eta_1=\xi_{f_1,v_1}$ and $\eta_2=\xi_{f_2,v_2}$, where $f_1,f_2 \in C_c(G)$ and $v_1,v_2 \in \mathcal{K}$. In this case we have: $\vert \left\langle \xi_{f_1,v_1}(t^{-1}x), \xi_{f_2,v_2} (x) \right\rangle_{\mathcal{K}}\vert=$

$=\left\vert \int\limits_{H} \int\limits_{H} \left\langle f_1(t^{-1}x h_1) \sigma(h_1)(v_1), f_2 (xh_2) \sigma(h_2)(v_2)\right\rangle_{\mathcal{K}} dh_1 dh_2 \right\vert \leq $

 $\leq  \int\limits_{H} \int\limits_{H} \vert f_1(t^{-1}x h_1) \vert \cdot \vert f_2 (xh_2) \vert \cdot \vert \vert v_1 \vert \vert_{\mathcal{K}} \cdot \vert \vert v_2\vert \vert _{\mathcal{K}} dh_1 dh_2 < C,$ where $C$ is a constant which does not depend on $t$, but depends on $\eta_1, \eta_2$. From here the conclusion follows. Note the last assertion of the lemma follows using the change of variables $y:=t^{-1}x$ and the positivity of $\rho$.
\end{proof}

\begin{remark}
\label{rem::evaluate_rho_function}
\ref{lem::first_reduction_2} can be used in the following way. In order to show that induced unitary representations are $C_0$, it is enough to evaluate integrals of the form $\int\limits_{t_n(f_1KH) \cap f_2KH} \left(\frac{\rho(t^{-1}_{n}x)}{\rho(x)}\right)^{1/2} d\mu(xH)$, where $f_1, f_2 \in G$ are considered to be fixed and $t_n \to \infty$.
\end{remark}


\begin{lemma}
\label{lem::H_left_cosets}
Let $K \leq G$ be open-compact. Let  $g,f_{1}, f_{2} \in G$ with $g(f_1KH) \cap f_2KH \neq \emptyset$. Then $g(f_1KH) \cap f_2KH= \sqcup_{i \in I}f_{2}k_iH$, for some $\{k_i\}_{i \in I} \subset K/(K \cap H)$ pairwise different. In addition, for every $i \in I$, there is a unique $k_{k_i} \in K/(K \cap H)$ and a unique $h_i \in H$ with $gf_{1}k_{k_i}= f_{2}k_i h_i \in f_{2}k_i H$. 
\end{lemma}
\begin{proof}
Let $x \in gf_{1}KH \cap f_{2}KH$. Then there exist $k,k' \in K$ and $h,h' \in H$, with $x=gf_{1}kh=f_{2}k'h'$; so $xh^{-1}=gf_{1}k=f_{2}k'h'h^{-1}$. By taking $k' \in K/(K \cap H)$ we obtain the first part of the lemma. Suppose there are $k,k' \in K/(K \cap H)$ and $h,h' \in H$ with $k \notin k'H$ and $gf_{1}k= f_{2}k_ih, gf_{1}k'=f_{2}k_ih' \in f_2k_iH$. From here we have $k=k'$ and $h=h'$. Note for $i \neq j \in I$, we might have $h_i=h_j$.
\end{proof}

\begin{lemma}
\label{lem::integral_evaluation}
Let $K \leq G$ be open-compact and $G$ be unimodular. Consider on $G/H$ a CGQIRB-measure $\mu$, with associated rho-function $\rho$ on $G$. Let  $g,f_{1}, f_{2} \in G$. Then there is a constant $C>0$, depending only on $K$, $\rho$ and $f_{1}, f_{2}$ with $$\int\limits_{gf_{1}KH \cap f_{2}KH} \left(\frac{\rho(g^{-1}x)}{\rho(x)}\right)^{1/2} d\mu(xH) \leq C  \int\limits_{\sqcup_{i \in I_{n}}f_{2}k_{i}H} \Delta_{H}(h_i)^{-1/2} d\mu(f_{2}k_iH)$$ where $I,k_i, h_i$ are given by~\ref{lem::H_left_cosets}.
\end{lemma}

\begin{proof}
Suppose $gf_{1}K H\cap f_{2}KH \neq \emptyset$, as otherwise the conclusion is trivial.

By~\ref{lem::H_left_cosets},  $gf_{1}K H\cap f_{2}KH= \sqcup_{i \in I}f_{2}k_iH$, for some $\{k_i\}_{i \in I} \subset K/(K \cap H)$ pairwise different. Let $x \in gf_{1}KH \cap f_{2}KH$. Then, by the same~\ref{lem::H_left_cosets}, $x= gf_{1}k_{k_i}h=f_{2}k_ih_ih$, for some $h \in H$ and some $i \in I$. Therefore, $\left(\frac{\rho(g^{-1}x)}{\rho(x)}\right)^{1/2}= \left(\frac{\rho(f_{1}k_{k_i}h)}{\rho(f_{2}k_ih_ih)}\right)^{1/2}= \left(\frac{\rho(f_{1}k_{k_i})\Delta_{H}(h)}{\rho(f_{2}k_i)\Delta_{H}(h_ih)}\right)^{1/2}$. As the map $\rho$ is continuous on $G$ and $K$ is compact, there exists a constant $C>0$ with $0 < \left(\frac{\rho(f_{1}k)}{\rho(f_{2}k')}\right)^{1/2} \leq C$, for every $k,k' \in K$. We obtain $\left(\frac{\rho(g^{-1}x)}{\rho(x)}\right)^{1/2} \leq C \Delta_{H}(h_i)^{-1/2}$, for $x \in f_{2}k_iH$. The conclusion follows.
\end{proof}

Note for $i \neq j \in I$, so for $f_{2}k_iH, f_{2}k_jH$, one can have $\Delta_{H}(h_i)^{-1/2}= \Delta_{H}(h_j)^{-1/2}$. Therefore, the function $\Delta_{H}(h_i)^{-1/2}$ might be integrated on a bigger subset than $f_{2}k_iH$, and thus on a subset that might not have measure zero. We resume below our general strategy to prove induced unitary representations on locally compact groups are $C_0$. 

\begin{remark}[The strategy: first step]
\label{rem::the_strategy_first_step}

Let $K \leq G$ be open-compact and $G$ be unimodular.  Consider on $G/H$ a CGQIRB-measure $\mu$, with associated rho-function $\rho$ on $G$. By~\ref{lem::K-inv_borel_measure} and~\ref{rem::equiv_induced_unit_rep}, $\mu$ and the associated rho-function $\rho$ are both $K$--invariant. From now on consider fixed these $\mu$ and $\rho$. As $K$ is open-compact, $0 \neq  \mu(KH) < \infty$. Suppose $\mu(KH)=1$. Let $(\sigma, \mathcal{K})$ be a unitary representation of $H$ and denote by $(\pi_{\sigma}, \mathcal{H_{\sigma}})$ the induced unitary representation on $G$. Note we have applied~\ref{rem::equiv_induced_unit_rep} and~\ref{notation_induced_rep}. Assume there exist a sequence $\{t_{n}\}_{n>0} $ of $G$ and $\eta_1,\eta_2 \in \mathcal{H_{\mu}}$ with $t_{n} \rightarrow \infty$ and $\vert \left\langle \pi_{\sigma}(t_{n})\eta_1, \eta_2 \right\rangle \vert \nrightarrow 0$, thus the representation $(\pi_{\sigma}, \mathcal{H_{\sigma}})$ is not $C_0$. To the sequence $\{t_n\}_{n>0}$ apply~\ref{lem::first_reduction_1} and then~\ref{lem::first_reduction_2}. By~\ref{rem::evaluate_rho_function} it is enough to evaluate the integrals $\int\limits_{t_n(f_{1}KH) \cap f_{2}KH} \left(\frac{\rho(t^{-1}_{n}x)}{\rho(x)}\right)^{1/2} d\mu(xH),$ where $f_1, f_2 \in G$ are fixed and $t_n \to \infty$. 

First of all,  fix $t_n$.  Apply~\ref{lem::H_left_cosets} and~\ref{lem::integral_evaluation} to $t_n,f_1,f_2$. One obtains $$\int\limits_{t_n(f_{1}KH) \cap f_{2}KH} \left(\frac{\rho(t^{-1}_{n}x)}{\rho(x)}\right)^{1/2} d\mu(xH) \leq C \int\limits_{\sqcup_{i \in I_{n}}f_{2}k_{i,n}H} \Delta_{H}(h_{i,n})^{-1/2} d\mu(f_{2}k_{i,n}H),$$
where the constant $C>0$ depends only on $K,\rho, f_1,f_2$; the set $I_{n}$ and $k_{i,n}, h_{i,n}$, with $i \in I_{n}$, depend on $t_n, f_1$ and $f_2$.  As noticed above, for $i \neq j \in I_n$, so for $f_{2}k_{i,n}H, f_{2}k_{j,n}H$, one can have $\Delta_{H}(h_{i,n})^{-1/2}= \Delta_{H}(h_{j,n})^{-1/2}$. Therefore, the function $\Delta_{H}(h_{i,n})^{-1/2}$ it might be integrated on a bigger subset than $f_{2}k_{i,n}H$, and thus on a subset that might not have measure zero. We want to  show that integral tends to zero when $n \to \infty$. This would be a contradiction of our assumption  $\vert \left\langle \pi_{\sigma}(t_{n})\eta_1, \eta_2 \right\rangle \vert \nrightarrow 0$.
\end{remark}


\begin{remark}
Let $K\leq G$ be open-compact and $G$ be unimodular. Suppose we have been able to evaluate the intersections $gKH \cap KH$, for $g \in G$.  It would remain to evaluate the values of $\Delta_{H}^{-1/2}$. These values strictly depend on the structure of the group $H$. Because of this, we restrict ourself to the case when $G$ is a closed subgroup of $\Aut(\T)$ and $H$ is a closed subgroup of $\Aut(\T)_{\xi}$,  with $\xi \in \partial \T$. In this case the values of the function $\Delta_{H}^{-1/2}$ are determined by the hyperbolic elements of $H$, the structure of those being very well understood. 
\end{remark}

\section{Vanishing results for the universal group $\G$}
\label{sub::vanishing_results}

In this section we consider parabolically induced unitary representations of the universal group $\G$. Recall by \cite{BM00a} $\G$ is unimodular when $F$ is primitive. We split this study in two parts:  when $H \leq \G_{\xi}$ does not contain hyperbolic elements, and the general case, when $H \leq \G_{\xi}$ does contain hyperbolic elements. By~\cite{Cio_rel} $\G_{\xi}^{0}$ is a closed, non-compact subgroup of $\G$, for every $\xi \in \bd \T$.

\subsection{The non-hyperbolic case}
\label{subsec::unimod_case}

\begin{remark}
\label{rem::H_unimodular}
Let $\xi  \in \bd \T$. If $H \leq \Aut(\T)_{\xi}$ is a closed subgroup not containing hyperbolic elements then $H$ is unimodular. This is because $H$ can be written as a countable union of compact subgroups. Indeed, by \cite{Ti70} as $H$ contains only elliptic elements, each element of $H$ fixes pointwise an infinite geodesic ray of $\T$ with endpoint $\xi$. Thus every element of $H$ is contained in some $H_x$ for some vertex $x$ of $\T$ and $H_x$ is compact (whence unimodular).
\end{remark}

\begin{lemma}
\label{lem::unimod_case_parabolic}
Let $x \in \T$ and $\xi \in \bd \T$. Let $K \leq \Aut(\T)_{x}$ be closed and let $H$ be a closed, non-compact subgroup of $\Aut(\T)_{\xi}$, not containing hyperbolic elements. Let  $g \in \Aut(\T)^{+}$. If $gKH \cap KH \neq \emptyset$, then there exists $k_{g} \in K$ with $gKH \cap KH \subset k_{g}K_{[x, x_{g}]} H,$ where $x_{g} \in [x, \xi)$ has the properties $\dist_{\T}(x,x_{g})=\frac{\dist_{\T}(x, g(x))}{2}$ and $k_{g}$ sends $[x,x_{g}]$ into the first half of the geodesic segment $[x,g(x)]$.
\end{lemma}

\begin{proof}
From $gKH \cap KH \neq \emptyset$, $g=k'hk$, for some $h \in H$ and $k',k \in K$. We want to determine the domain in $K$ of the variable $k'$. From $g=k'h k$ we have:
\begin{equation}
\label{eq::equation_0}
\dist_{\T}(x, g(x))=\dist_{\T}(x, k' h(x))= \dist_{\T}(x, h(x)).
\end{equation}

As $h$ is not hyperbolic, denote by $x_{h}$ the first vertex of the geodesic ray $[x, \xi)$ fixed by $h$. From equation~(\ref{eq::equation_0}) we obtain $\dist_{\T}(x,x_{h})=\frac{\dist_{\T}(x, g(x))}{2}$, thus $x_{h}$ is a precise point on the geodesic ray $[x, \xi)$ determined only by the element $g$ and not by the non-hyperbolic element $h$. So take $x_g:=x_h$. Because $k'([x,h(x)])= [x, g(x)]$, $k'$ sends the geodesic segment $[x,x_{g}]$ into the first half of the geodesic segment $[x,g(x)]$. We conclude $k' \in k_{g} K_{[x,x_{g}]}$, where $k_{g} \in K$ is a fixed element sending $[x,x_{g}]$ into the first half of the geodesic segment $[x,g(x)]$. 
\end{proof}

\begin{theorem}
\label{thm::vanishing_unimodular_case}
Let $\xi \in \bd \T$, $x \in \T$, $G$ be a closed, non-compact, unimodular subgroup of $\Aut(\T)^+$ and suppose the index in $K:=G_x$ of $G_{[x,\xi]}$ is infinite. Let $H$ be a closed, non-compact subgroup of $G_{\xi}$, not containing hyperbolic elements and let $(\sigma, \mathcal{K})$ be a unitary representation of $H$. Then the induced unitary representation $(\pi_{\sigma}, \mathcal{H_{\sigma}})$ on $G$ is $C_0$.
\end{theorem}

\begin{proof}

By~\ref{rem::equiv_induced_unit_rep} and because $H, G$ are unimodular, it is enough to consider the case when the rho-function $\rho$ is the constant function $\mathbf{1}$ on $G$. Thus, the measure $\mu$ on $G/H$ associated with the rho-function $\mathbf{1}$ on $G$ is $G$--invariant. As $K$ is open and compact with respect to the locally compact topology on $G$, we have $0 \neq \mu(KH) < \infty$. Assume there exist a sequence $\{t_{n}\}_{n>0} $ of $G$ and $\eta_1,\eta_2 \in \mathcal{H_{\sigma}}$ with $t_{n} \rightarrow \infty$ and $\vert \left\langle \pi_{\sigma}(t_{n})\eta_1, \eta_2 \right\rangle \vert \nrightarrow 0$. To the sequence $\{t_n\}_{n>0}$ apply~\ref{lem::first_reduction_1} and then~\ref{lem::first_reduction_2}. Moreover, by~\ref{rem::evaluate_rho_function} it is enough to evaluate $ \mu(t_n(h_{i}KH) \cap h_{j}'KH)$, where $h_i, h_j'$ are considered to be fixed and $t_{n} \rightarrow \infty$.  Note $ \mu(t_n(h_{i}KH) \cap h_{j}'KH)=\mu((h_{j}')^{-1}t_n h_{i}KH \cap KH)$. 

If $(h_{j}')^{-1}t_n h_{i}KH \cap KH \neq \emptyset $ apply~\ref{lem::unimod_case_parabolic} to $g_n:= (h_{j}')^{-1}t_n h_{i}$. We obtain $g_n KH \cap KH \subset k_{g_n}G_{[x, x_{g_n}]} H,$ where $x_{g_n} \in [x, \xi)$ with one of the properties being $\dist_{\T}(x,x_{g_n})=\frac{\dist_{\T}(x, g_n(x))}{2}$. As $t_n \to \infty$, we also have $g_n \to \infty$ ($h_i, h_j'$ being fixed); in addition, $\dist_{\T}(x,x_{g_n}) \rightarrow \infty $ when $n \rightarrow \infty$. To evaluate $\mu(g_n KH \cap KH)$ it is enough to compute $\mu(k_{g_n}G_{[x, x_{g_n}]} H)= \mu(G_{[x, x_{g_n}]} H)$, where $\dist_{\T}(x,x_{g_n}) \xrightarrow[g_n \to\infty]{}  \infty $.  We claim $\mu(G_{[x, x_{g_n}]} H) \xrightarrow[g_n \to\infty]{}  0$, giving a contradiction. Indeed, there are two cases that should be considered: either for every $y \in (x, \xi)$ the index of $G_{[x,y]}$ in $K$ is finite or there exists $y \in (x, \xi)$ with the index of $G_{[x,y]}$ in $K$ is infinite. Consider the first case; so the index in $K$ of $G_{[x, x_{g_n}]}$ is finite for every $g_n$. Moreover, as $[K: G_{[x, \xi]}]=\infty$, $[K: G_{[x, x_{g_n}]}]  \xrightarrow[g_n \to\infty]{}  \infty$. As $\mu(KH)< \infty$, $\mu$ is $G$--invariant, and so $K$--invariant, the claim follows. Consider the second case; so there exists $N>0$ such that for every $n \geq N$ we have the index of $G_{[x, x_{g_n}]}$ in $K$ is infinite. By~\ref{lem::compact_open_quotient} applied to $K'= G_{[x, x_{g_n}]}$, we have $\mu(G_{[x, x_{g_n}]} H)=0$, for every $n \geq N$. The theorem follows.

\end{proof}

\begin{corollary}
\label{cor::vanishing_unimod_case}
Let $F$ be primitive and let $\xi \in \bd \T$. Let $H$ be a closed, non-compact subgroup of $\G_{\xi}$, not containing hyperbolic elements and let $(\sigma, \mathcal{K})$ be a unitary representation of $H$. Then the induced unitary representation $(\pi_{\sigma}, \mathcal{H_{\sigma}})$ on $\G$ is $C_0$.
\end{corollary}

\begin{proof}
The hypotheses of~\ref{thm::vanishing_unimodular_case} are fulfilled: let $K:=\G_x$ and by~\ref{lem::infinit_index} applied to $\G_{\xi}^{0}$ we have $[K: K_{[x, \xi]}] = \infty$.
\end{proof}

\subsection{The hyperbolic case}
\label{subsec::general_case}

Let $\xi \in \bd \T$. In this subsection we consider $H$  a closed subgroup of $\Aut(\T)_{\xi}$ containing hyperbolic elements. This implies $H$ is not compact. 

\subsubsection{Structure and modular function of parabolic subgroups}

\begin{lemma}
\label{lem::hyp_H}
Let $\xi \in \bd \T$ and $H \leq \Aut(\T)_{\xi}$ be a closed subgroup containing hyperbolic elements. Then there exists a hyperbolic element $\gamma \in H$, of attracting endpoint $\xi$, that is minimal, in the sense any other hyperbolic element $\gamma' \in H$ is written $\gamma' =\gamma^{n} h$, where $n \in \mathbb{Z}$, $\vert n\vert \vert \gamma \vert =\vert \gamma' \vert$ and $h \in H \cap \Aut(T)_{\xi}^{0}$. 
\end{lemma}

\begin{proof}
 Let $Hyp(H):=\{\gamma \in H \; \vert \; \gamma \text{ is hyperbolic}\}$. Let $hyp_{H}:= \min_{\gamma \in H}(\vert \gamma \vert)$. Note $hyp_{H}$ exists and $hyp_{H} \geq 1$. Let fix $\gamma \in H$ with $\vert \gamma \vert = hyp_{H}$. Fix also a vertex $x$ in $\Min(\gamma)$. Moreover, consider the attracting endpoint of $\gamma$ is $\xi$; if not take $\gamma^{-1}$. Let $\gamma' \in Hyp(H)$ and let $x_{\gamma'}$ be the first vertex of $[x, \xi)$ contained in $\Min(\gamma')$. By minimality $\vert \gamma' \vert $ is a multiple of $\vert \gamma \vert$. If the attracting endpoint of $\gamma'$ is $\xi$, then $\gamma^{-n}\gamma'(x_{\gamma})=x_{\gamma'}$, where $n \vert \gamma \vert= \vert \gamma' \vert$. Thus, $\gamma' =\gamma^{n} h$, where $h \in H_{x_{\gamma'}}$. If $\gamma'$ has $\xi$ as a repelling endpoint, then $\gamma^{n} \gamma' ((\gamma')^{-1}(x_{\gamma'}))=(\gamma')^{-1}(x_{\gamma'})$, where $n\vert \gamma \vert= \vert \gamma' \vert $. Thus $\gamma' =\gamma^{-n} h$, where now $h$ is in $H_{(\gamma')^{-1}(x_{\gamma'})}$.
\end{proof}

\begin{lemma}
\label{lem::modular_fct_H}
Let $\xi \in \bd \T$ and $H \leq \Aut(\T)_{\xi}$ be a closed subgroup containing hyperbolic elements. Let $\gamma$ be a hyperbolic element of $H$ with attracting endpoint $\xi$ and $x$ be a vertex of $\Min(\gamma)$. Then $ \frac{1}{(d-1)^{\vert \gamma \vert}} \leq\Delta_{H}(\gamma)=\frac{1}{[H_{\gamma(x)}: H_{x}]} \leq1$. In particular, $H$ is unimodular if and only if $H_{x}= H_{y}$, for every $y \in \Min(\gamma)$.
\end{lemma}

\begin{proof}
By~\ref{rem::H_unimodular}, for every $h \in H \cap \Aut(\T)_{\xi}^{0}$, $\Delta_{H}(h)=1$.  Note the following facts. Firstly, $H_{x} = H_{[x, \xi]}  \leq H_{\gamma(x)}$ are compact subgroups and  secondly, the index $[H_{\gamma(x)}: H_{x}] \leq (d-1)^{\vert \gamma \vert}$, where $d$ is the regularity of the tree $\T$. Moreover, $H_{\gamma(x)}= \gamma H_{x} \gamma^{-1}$. Let $dh$ denote the left Haar measure on $H$. Then $dh(H_{\gamma(x)})=dh(\gamma H_{x} \gamma^{-1})= dh(H_{x} \gamma^{-1})=\Delta_{H}(\gamma^{-1}) dh(H_{x}).$ As $\Delta_{H}(h)=1$ for $h \in H \cap \Aut(\T)_{\xi}^{0}$, we have $dh(H_{\gamma(x)})=dh(H_{x}) \cdot [H_{\gamma(x)}: H_{x}]$. From the above two equalities we obtain $1\leq \Delta_{H}(\gamma^{-1})=[H_{\gamma(x)}: H_{x}] \leq (d-1)^{\vert \gamma \vert}.$ In particular, $1=\Delta_{H}(\gamma^{-1})=[H_{\gamma(x)}: H_{x}]$ if and only if $H_{x}= H_{y}$, for every $y \in \Min(\gamma)$; thus $H_{y}= H_{\xi_{-}}$ for every $y \in \Min(\gamma)$, where $\xi_{-}$ is the repelling endpoint of $\gamma$.
\end{proof}

\begin{lemma}
\label{lem::index_fct_H}
Let $\xi \in \bd \T$ and $H \leq \Aut(\T)_{\xi}$ be a closed subgroup containing hyperbolic elements. Let $\gamma$ be a hyperbolic element of $H$ with attracting endpoint $\xi$ and $x$ be a vertex of $\Min(\gamma)$. Then  we have the following properties:
\begin{list}{\arabic{qcounter})~}{\usecounter{qcounter}}
\item
\label{lem::index_fct_H_one}
$[H_{\gamma(x)}: H_{x}]= [H_{\gamma^{2}(x)}: H_{\gamma(x)}]$;
\item
\label{lem::index_fct_H_two}
$[H_{\gamma^{n}(x)}: H_{x}]= [H_{\gamma^{n-m}(x)}: H_{x}] \cdot [H_{\gamma^{m}(x)}: H_{x}]$ for every $0 \leq m \leq  n$;
\item
\label{lem::index_fct_H_three}
$[H_{\gamma^{m}(x)}: H_{x}] \leq [H_{\gamma^{n}(x)}: H_{x}]$ for every $0 \leq m \leq  n$.
\end{list}
\end{lemma}

\begin{proof}
Note assertion~\ref{lem::index_fct_H_three}) is a consequence of assertion~\ref{lem::index_fct_H_two}) and the latter one follows from assertion~\ref{lem::index_fct_H_one}) and from $H_{x} \leq H_{\gamma^{m}(x)} \leq H_{\gamma^{n}(x)}$, for $0 \leq m \leq  n$. The first assertion follows as $H_{\gamma(x)}= \gamma H_{x} \gamma^{-1}$ and $H_{\gamma^{2}(x)}= \gamma H_{\gamma(x)} \gamma^{-1}$. Moreover, for every coset $h H_{x}$ of $H_{\gamma(x)}/H_{x} $ we have $\gamma h \gamma^{-1} \gamma H_{x} \gamma^{-1}$ is a coset of $H_{\gamma^{2}(x)} / H_{\gamma(x)} $ and vice versa. The lemma is proved. \end{proof}

For $F \leq \Sym\{1,\cdots, d\}$ and $e$ an edge of $\T$ we abuse notation and use $F_e $ to denote the stabiliser in $F$ of the colour from $\{1,\cdots, d \}$ of the edge $e$.
\begin{lemma}
\label{lem::hyp_element}
Let $F$ be transitive and let $\gamma \in \G$ be hyperbolic. Denote by $\xi_{+}, \xi_{-} \in \partial \T$ the attracting and respectively, the repelling endpoints of $\gamma$. Take $x \in (\xi_{-}, \xi_{+})$, the edges $e_-,e_+$ in the star of $x$ with $e_+ \in [x, \xi_+), e_- \in (\xi_-,x]$ and $K:=\G_{x}$. Then we have:
\begin{list}{\arabic{qcounter})~}{\usecounter{qcounter}}
\item
\label{lem::hyp_element_one}
$[\G_{[\gamma(x), \xi_{+}]}: \G_{[x, \xi_{+}]}]= \Delta_{\G_{\xi_{+}}}(\gamma^{-1})=\frac{[K: \G_{[x, \gamma^{-1}(x)]}] \cdot k_1}{d}$, where $d$ is the regularity of $\T$ and $k_1$ is the number of orbits of the edge $e_-$ in $\{1, \cdots,d\}$ under the stabilizer subgroup $F_{e_+} \leq F$;
\item
\label{lem::hyp_element_two}
$[\G_{[x, \xi_{-}]}: \G_{[\gamma(x), \xi_{-}]}]= \Delta_{\G_{\xi_{-}}}(\gamma)=\frac{[K: \G_{[x, \gamma(x)]}] \cdot k_2}{d}$  where  $k_2$ is the number of orbits of the edge $e_+$ in $\{1, \cdots,d\}$ under the stabilizer subgroup $F_{e_-} \leq F$;
\item
\label{lem::hyp_element_three}
$[K: \G_{[\gamma^{-1}(x),x]}]=[K: \G_{[x, \gamma(x)]}]= \frac{[\G_{[\gamma(x), \xi_{+}]}: \G_{[x, \xi_{+}]}] \cdot d}{k_1}=$

$= \frac{[\G_{[x, \xi_{-}]}: \G_{[\gamma(x), \xi_{-}]}] \cdot d}{k_2}$.
\end{list}
\end{lemma}

\begin{proof}
First we it is easy to see that 
\begin{equation}
\label{equ::stab_segment}
\gamma \G_{[\gamma^{-1}(x), x]} \gamma^{-1}= \G_{\gamma([\gamma^{-1}(x), x])}= \G_{[x, \gamma(x)]}. 
\end{equation}

Let $m$ be the left Haar measure on $\G$. Then we have 
\begin{equation}
\label{equ::left_right_gamma}
m(K)= m(\G_{[x, \gamma(x)]}) \cdot [K: \G_{[x, \gamma(x)]}]= m(\G_{[\gamma^{-1}(x), x]}) \cdot [K: \G_{[\gamma^{-1}(x), x]}].
\end{equation}

By a standard computation we have $ m(\G_{[x, \gamma(x)]})= \Delta_{\G}(\gamma^{-1}) m(\G_{[\gamma^{-1}(x), x]})$. As $\G$ is unimodular we obtain
\begin{equation}
\label{equ::neg_pos}
\Delta_{\G}(\gamma^{-1}) \cdot [K: \G_{[x, \gamma(x)]}]=[K: \G_{[x, \gamma(x)]}]= [K: \G_{[\gamma^{-1}(x),x]}].
\end{equation}

Let us prove assertion~\ref{lem::hyp_element_one}) of the lemma. First, by~\ref{lem::modular_fct_H} applied to $\G_{\xi_+}$ we have $[\G_{[\gamma(x), \xi_{+}]}: \G_{[x, \xi_{+}]}]= \Delta_{\G_{\xi_{+}}}(\gamma^{-1})=[\G_{[x, \xi_{+}]}: \G_{[\gamma^{-1}(x), \xi_{+}]}]$. As $\gamma$ is translating along the axis $(\xi_-, \xi_+)$, $F_{\gamma(e_+)}$ is isomorphic to $F_{e_+}$, thus the number of $F_{\gamma(e_+)}$-orbits of $\gamma(e_-)$ in $\{1, \cdots, d\}$ is the same as the $F_{e_+}$-orbits of $e_-$ in $\{1, \cdots, d\}$, which is $k_1$. As $F$ is transitive on $\{1, \cdots, d\}$, we have $[K: \G_{[\gamma^{-1}(x),x]}]=d \cdot [ \G_{e_-}: \G_{[\gamma^{-1}(x),x]}]$. Also as $\G$ has Tits' independence property \cite{BM00a,Amann}  $[\G_{[x, \xi_{+}]}: \G_{[\gamma^{-1}(x), \xi_{+}]}]=k_1 \cdot [ \G_{e_-}: \G_{[\gamma^{-1}(x),x]}]$. We conclude indeed $$[\G_{[\gamma(x), \xi_{+}]}: \G_{[x, \xi_{+}]}]=\frac{[K: \G_{[x, \gamma^{-1}(x)]}] \cdot k_1}{d}$$ and part~\ref{lem::hyp_element_one}) of the lemma is proved. The assertion~\ref{lem::hyp_element_two}) of the lemma goes in the same way. The assertion~\ref{lem::hyp_element_three}) of the lemma is a consequence of assertions~\ref{lem::hyp_element_one}),~\ref{lem::hyp_element_two}) and relation~(\ref{equ::neg_pos}).
\end{proof}

\subsubsection{The evaluation of $gKH \cap KH$}
\label{subsec::evaluation_gKH_cap_KH}
By~\ref{rem::the_strategy_first_step}, the next step is the evaluation of $gKH \cap KH$. This is because we need to integrate the modular function $\Delta_{H}^{-1/2}$ on the intersection $gf_{1}K H\cap f_{2}KH= \sqcup_{i \in I}f_{2}k_{i}H$, for $g,f_1,f_2 \in G$. We are able to evaluate $gKH \cap KH$  for the universal group $\G$ and not in a more general case. This is due to the $KA^{+}K$ decomposition of $\G$ proven in~\ref{lem::KAK_decomposition}, making our task easier. That decomposition might not hold in a more general situation. Using the $KA^{+} K$ decomposition, we only evaluate $gKH \cap KH$ when $g \in A^{+}$. This is given by the next technical proposition. We state the proposition as general as possible, making use of the following general definition. 

\begin{definition}
\label{def::proj_A+}
Let $G$ be a closed subgroup of $\Aut(\T)^{+}$, $x \in \T$ and $e$ be an edge of the star of $x$. Set $K:= G_{x}$ and define
$$A^{+}:=\{\gamma \in G \; \vert \;  e \subset \Min(\gamma), \gamma(e) \subset  \T_{x,e}  \} \cup \{\id\}.$$
Let  $\xi$ be an endpoint in $\bd \T_{x,e}$. 
Define the map $ \proj_{(x,\xi]} : A^{+} \to (x, \xi]$ by $\proj_{(x,\xi]}(g)$ is the vertex or the endpoint $\xi$ with the property $[x, \xi_{g,+}] \cap [x,\xi]=[x, \proj_{(x,\xi]}(g)]$, where $ \xi_{g,+}$ is the attracting endpoint of $g$. As $g \in A^{+}$, note $ \proj_{(x,\xi]}(g)$ is indeed a point  in $(x, \xi]$. Let now $g \in G$ be a hyperbolic element translating the vertex $x$. Consider its $K$--double coset $KgK$ and set $\proj_{(x,\xi]}(KgK):=\max_{g' \in A^{+} \cap KgK}\{\proj_{(x, \xi]}(g')\}$. 
\end{definition}

\begin{proposition}
\label{prop::evaluation_on_A}
Let $G$ be a closed subgroup of $\Aut(\T)^{+}$ and let $\xi \in \partial \T$. Assume $G_{\xi}$ contains hyperbolic elements. Let $H < G_{\xi}$ be a closed subgroup containing also hyperbolic elements. Let $\gamma$ be a minimal hyperbolic element of $H$ given by~\ref{lem::hyp_H}, with attracting endpoint $\xi$, and let $x$ be a vertex of $\Min(\gamma)$. Set $K:=G_{x}$. Choose the edge $e$ in the star of $x$ and define $A^{+}$ such that $\gamma \in A^{+}$. 

Let $g \in A^{+}$. Assume  $\proj_{(x,\xi]}(KgK)=\proj_{(x,\xi]}(g)$. Assume there also exist $k_2 \in K \setminus \{H\cap K\}, k_1 \in K$ and $h \in H$ with $k_1gk_2=h=\gamma^{n}h_0$, where $h_0 \in H\cap G_{\xi}^{0}$ and $n \in \ZZ$.

Then $0 \leq \vert n \vert \leq \frac{\dist_{\T}(x,g(x))}{\vert \gamma \vert} $ and $k_1 \in G_{[x,x_h]}$, where $x_h \in [x, \proj_{(x,\xi]}(g)]$ is with $\dist_{\T}(x, x_h)= \frac{\dist_{\T}(x,g(x))+ \sign(n) \vert \gamma^{n} \vert}{2}$, where $\sign(0)=0$.
\end{proposition}

\begin{proof}
Denote by $e$ the edge of the star of $x$ with $\xi \in \partial \T_{x,e}$. In particular, $A^{+}$ is defined using $x$ and $e$. Let $\xi_{+}$ and $\xi_{-}$ be the attracting and the repelling endpoints of $g$. As $k_2$ is not fixing $\xi$, we denote $x_{k_2}$ the vertex of the geodesic line $(\xi_{-}, \xi)$ with the property $[x, k_{2}(\xi)) \cap (\xi_{-}, \xi)=[x, x_{k_2}]$. We have three cases: either $x_{k_2} \in [x, \xi)$ or $x_{k_2} \in (x, g^{-1}(x))$ or $x_{k_2} \in [g^{-1}(x), \xi_{-})$.

Suppose $x_{k_2} \in [x, \xi)$. Because $k_1gk_2(\xi)=\xi$, $k_1gk_2(e)$ is an edge of $\T_{x,e}$ and the orientation of $k_1gk_2(e)$ induced from $e$ points towards the boundary $ \bd \T_{x,e}$, like $e$. Therefore,  $k_1gk_2 \in A^{+}$. As $k_1gk_2 \in H$, we have $h=k_1gk_2 \in A^{+} \cap H$. As by hypothesis $\proj_{(x,\xi]}(KgK)=\proj_{(x,\xi]}(g)$, we conclude $g \in A^{+} \cap H$. In addition, by \ref{lem::hyp_H} we have $h= \gamma^{n}h_0$, where $h_0 \in K \cap H$; thus $\vert g \vert=\vert h \vert=n \vert  \gamma \vert$. As $k_1gk_2(\xi)=\xi$ and because $g$ is hyperbolic, with attracting endpoint $\xi$ and with $x \in \Min(g)$, $k_1$ must fix at least the vertex $g(x) \in (x, \xi)$. Therefore, $k_1 \in G_{[x, x_h]}$, where $x_h=g(x)$ and the conclusion follows.

Suppose $x_{k_2} \in [g^{-1}(x), \xi_{-})$. Then $gk_2(e)$ is an edge of $\T_{x,e}$ and the orientation of $gk_2(e)$ induced from $e$ points outwards the boundary $\bd T_{x,e}$, thus towards $e$. Because $x_{k_2} \in [g^{-1}(x), \xi_{-})$, by applying $k_1$ to $gk_2$, $k_1(\T_{x,e}) \cap \T_{x,e}=\{x\}$ and the edge $k_1gk_2(e)$ points towards the edge $e$. Therefore $k_1gk_2$ must be a hyperbolic element (of $H$) translating the vertex $x$ outwards the half-tree $ \T_{x,e}$. Consequently, $\xi$ is the repelling endpoint of $k_1gk_2$, as $k_1gk_2(\xi)=\xi$. Otherwise saying, $\xi$ is the attracting endpoint of the hyperbolic element $(k_1gk_2)^{-1}= h^{-1} \in H$ and $x \in \Min(h^{-1})$. We have $\vert h \vert= \vert h^{-1}\vert= \dist_{\T}(x, g(x))=\vert n \vert \vert \gamma \vert$. Although we can say more, we do not impose any restriction for $k_1$, so $k_1 \in G_{[x,x_h]}$ where $x_h=x$. The conclusion of the proposition is still valid in this case. 

Suppose now $x_{k_2} \in (g^{-1}(x),x)$. We claim $ g(x_{k_2}) \in [x,\proj_{(x,\xi]}(g)]$. Indeed, supposing the contrary we have $\proj_{(x,\xi]}(g) \in (x,g(x_{k_2}))$. Then $g(x_{k_2}) \notin [x, \xi)$. As the geodesic ray $[x_{k_2}, k_2(\xi))$ is sent by $g$ into the geodesic ray $[g(x_{k_2}), gk_2(\xi))$,  $[g(x_{k_2}), gk_2(\xi))$ does not intersect $[x, \xi)$. However, by applying $k_1$, we must have $k_1g(x_{k_2}) \in [x, \xi)$, as $k_1gk_2(\xi)=\xi$. This is a contradiction with $\proj_{(x,\xi]}(KgK)=\proj_{(x,\xi]}(g)$ and the claim follows. As $k_1gk_2(\xi)=\xi$, from the latter claim we immediately have $k_1 \in G_{[x, g(x_{k_2})]}$.  From here we deduce the following two facts:
\begin{list}{\arabic{qcounter})~}{\usecounter{qcounter}}
\item 
\label{list::orientation_reversed}
the segment $[x, k_2^{-1}(x_{k_2}))$, where $k_2^{-1}(x_{k_2}) \in (x, \xi)$, is sent by $h=k_1gk_2$ into the segment $(g(x_{k_2}),k_1g(x)] \subset \T_{x,e} \setminus \{[x, \xi)\}$, and the orientation is reversed;
\item 
the edge $k_1gk_2(e)$ belongs to $\T_{x,e}$ and the orientation of $k_1gk_2(e)$ induced from $e$ would point outwards the boundary $\bd T_{x,e}$, thus towards $e$. Therefore, either $k_1gk_2$ is elliptic, or $k_1gk_2$ is hyperbolic in $H$, with translation length strictly smaller than $\dist_{\T}(x, g(x))$. 
\end{list}

Our next claim is $h$ is elliptic if and only if $\dist_{\T}(x,x_{k_2})= \frac{\dist_{\T}(x,g^{-1}(x))}{2}$. Suppose $h=k_1gk_2$ is elliptic. Then by the above fact~\ref{list::orientation_reversed}) we know the segment $h([x, k_2^{-1}(x_{k_2})))$ does not intersect $[x,\xi)$. As $h \in H$ is elliptic, $h$ fixes the midpoint of the segment $[k_2^{-1}(x_{k_2}), h(k_2^{-1}(x_{k_2}))]=[k_2^{-1}(x_{k_2}), g(x_{k_2})]$. We deduce $k_2^{-1}(x_{k_2})=h(k_2^{-1}(x_{k_2}))=g(x_{k_2})$, from where  $\dist_{\T}(x,x_{k_2})= \frac{\dist_{\T}(x,g^{-1}(x))}{2}$. Suppose now $\dist_{\T}(x,x_{k_2})= \frac{\dist_{\T}(x,g^{-1}(x))}{2}$, so we need to prove $h$ is elliptic. Indeed, $k_2^{-1}(x_{k_2})= g(x_{k_2})$. As $k_1 \in G_{[x, g(x_{k_2})]}$, we conclude  $h(k_2^{-1}(x_{k_2}))=k_1gk_2(k_2^{-1}(x_{k_2}))=g(x_{k_2})=k_2^{-1}(x_{k_2})$, so $h$ is elliptic.  The equivalence follows. 

For $h$ elliptic, we resume the following: $k_1gk_2 \in H \cap G_{\xi}^{0}$, so $n=0$, and $k_1 \in G_{[x, x_h]}$, where $x_h:= g(x_{k_2}) \in [x,\proj_{(x,\xi]}(g)]$, with $\dist_{\T}(x,g(x_{k_2})) = \frac{\dist_{\T}(x,g(x))}{2}$. If $h=k_1gk_2$ is hyperbolic, then $\dist_{\T}(x,x_{k_2}) \neq \frac{\dist_{\T}(x,g^{-1}(x))}{2}$. 

Suppose $\dist_{\T}(x,x_{k_2}) < \frac{\dist_{\T}(x,g^{-1}(x))}{2}$, this implies $\frac{\dist_{\T}(x,g(x))}{2} < \dist_{\T}(x,g(x_{k_2})).$ Moreover, using the above fact~\ref{list::orientation_reversed}) and $h$ is hyperbolic fixing $\xi$, we conclude $\xi$ is the attracting endpoint of $h$ and $h$ translates the vertex $k_2^{-1}(x_{k_2}) \in (x, \xi)$ to $k_1g(x_{k_2})= g(x_{k_2}) \in (x, \xi)$. By \ref{lem::hyp_H}, $h= \gamma^{n} h_0$, for some $h_0 \in  H\cap G_{\xi}^{0}$, and $n$ is such that $n \vert  \gamma \vert= \vert h \vert = \dist_{\T}(x, g(x))-2 \dist_{\T}(x, k_2^{-1}(x_{k_2})) < \dist_{\T}(x, g(x))$. In addition, $k_1 \in G_{[x,x_h]}$, where $x_h= g(x_{k_2}) \in [x, \proj_{(x,\xi]}(g)]$ and indeed $\dist_{\T}(x, x_h)= \frac{\dist_{\T}(x,g(x))- \vert \gamma^{n} \vert}{2} + \vert \gamma^{n} \vert$.

Suppose now  $\dist_{\T}(x,x_{k_2}) > \frac{\dist_{\T}(x,g^{-1}(x))}{2}$, this implies $\frac{\dist_{\T}(x,g(x))}{2} > \dist_{\T}(x,g(x_{k_2})).$ 	As before, using the above fact~\ref{list::orientation_reversed}) and $h$ is hyperbolic fixing $\xi$, we conclude $\xi$ must be the repelling endpoint of $h$ and $h^{-1}$ translates the vertex $g(x_{k_2}) \in \Min(h) \cap (x, \xi)$ to $k_2^{-1}(x_{k_2}) \in (x, \xi)$. By \ref{lem::hyp_H}, we have that $h= \gamma^{-n} h_0$, for some $h_0 \in  H\cap G_{\xi}^{0}$, and $n>0$ is such that $n \vert  \gamma \vert= \vert h \vert = \dist_{\T}(x, g(x))-2 \dist_{\T}(x, g(x_{k_2}) ) < \dist_{\T}(x, g(x))$. In addition, $k_1 \in G_{[x,x_h]}$, where $x_h= g(x_{k_2}) \in [x, \proj_{(x,\xi]}(g)]$ and indeed $\dist_{\T}(x, x_h)= \frac{\dist_{\T}(x,g(x))- \vert \gamma^{-n} \vert}{2} $. The proposition is proven.
\end{proof}

When $H$ is unimodular, we obtain the following.
\begin{corollary}
\label{cor::evaluation_on_A_H_unimodular}
Let $G$ be a closed subgroup of $\Aut(\T)^{+}$ and let $\xi \in \partial \T$. Assume $G_{\xi}$ contains hyperbolic elements. Let $H < G_{\xi}$ be a closed, unimodular, subgroup containing also hyperbolic elements. Let $\gamma$ be a minimal hyperbolic element of $H$ given by~\ref{lem::hyp_H}, with attracting endpoint $\xi$, and let $x$ be a vertex of $\Min(\gamma)$. Set $K:=G_{x}$. Choose the edge $e$ in the star of $x$ such that $\gamma \in A^{+}$.

Let $g \in A^{+}$. Assume $\proj_{(x,\xi]}(KgK)=\proj_{(x,\xi]}(g)$. Assume there also exist $k_2 \in K \setminus \{H\cap K\}, k_1 \in K$ and $h \in H$ with $k_1gk_2=h=\gamma^{n}h_0$, where $h_0 \in H\cap G_{\xi}^{0}$ and $n \in \ZZ$. Then $h_{0} \in K \cap H$ and $\vert n \vert= \frac{\dist_{\T}(x,g(x))}{\vert \gamma \vert}$.  If $n>0$ then $k_1 \in G_{[x,x_h]}$, where $x_h \in [x, \xi]$ is with $\dist_{\T}(x, x_h)= \dist_{\T}(x,g(x))$. If $n<0$ then $k_1^{-1} \in kG_{[x,\gamma^{n}(x)]}$, where $k \in K$ with $[x, g(x)]=k([x, \gamma^{n}(x)])$.
\end{corollary}

\begin{proof}
We keep all the notation from the proof of Proposition~\ref{prop::evaluation_on_A}. As $H$ is unimodular, by Lemma~\ref{lem::modular_fct_H} we have $H_x= H_{y}= H_{\xi_{-}}$, for every $y \in \Min(\gamma)$, where $\xi_{-} \in \bd \T$ is the repelling endpoint of $\gamma$. This proves  $h_0 \in K \cap H$ and so $\dist_{\T}(x,h(x))=\dist_{\T}(x,g(x))=\dist_{\T}(x,\gamma^{n}(x))=\vert n \vert \vert \gamma \vert$. If $n>0$, by applying Proposition~\ref{prop::evaluation_on_A}, we directly obtain $k_1 \in G_{[x,x_h]}$, where $x_h \in [x, \proj_{(x,\xi]}(g)]$ is with $\dist_{\T}(x, x_h)= \dist_{\T}(x,g(x))$. It remains the case $n<0$. By the proof of Proposition~\ref{prop::evaluation_on_A}, the case $n <0$ with $\dist_{\T}(x,g(x))=\vert n \vert \vert \gamma \vert$ can occur only when $x_{k_2} \in [g^{-1}(x), \xi_{-})$. Let us compute $g(\xi)= (k_1)^{-1} \gamma^{n}h_0 (k_2)^{-1}(\xi)$. As $x_{k_2} \in [g^{-1}(x), \xi_{-})$, we have $(k_2)^{-1}(\xi) \notin \partial \T_{x,e}$. Then $h_0 (k_2)^{-1}(\xi)$ is still a point in $\{\partial \T \setminus \partial  \T_{x,e}\}$ as $h_0 \in H_{[x, \xi]}$. By applying $\gamma^{n}$ to $h_0 (k_2)^{-1}(\xi)$ and because $n<0$ we have $\gamma^{n}h_0 (k_2)^{-1}(\xi)$ in $\{\partial \T \setminus \partial  \T_{\gamma^{n}(x),\gamma^{n}(e)}\}$. Note $g(\xi) \in \T_{g(x),g(e)} \subsetneq \T_{x,e}$, as $g \in A^{+}$. By applying $k_1$ to $g(\xi)$ we must have $k_1([x, g(x)])=[x, \gamma^{n}(x)]$. We obtain $k_1^{-1} \in k \G_{[x, \gamma^{n}(x)]}$, where $k \in K$ with $[x, g(x)]=k([x, \gamma^{n}(x)])$.

\end{proof}

\subsubsection{The proof}

We are now ready to prove parabolically induced unitary representations on the universal group $\G$, induced from closed subgroups $H \leq \G_{\xi}$  containing hyperbolic elements, are $C_0$. We distinguish two cases: either $H$ is unimodular or $H$ is not unimodular.

\begin{remark}[The strategy: second step]
\label{rem::strategy_second_step}

Let $G$ be a closed subgroup of $\Aut(\T)^{+}$,  $\xi \in \partial \T$ and $H$ be a closed subgroup of $G_{\xi}$ containing hyperbolic elements. Applying~\ref{rem::the_strategy_first_step} it remains to integrate the modular function $\Delta_{H}^{-1/2}$ on the intersection $t_n(f_{1}K H)\cap f_{2}KH= \sqcup_{i \in I_n}f_{2}k_{i,n}H$, for $t_n,f_1,f_2 \in G$. In order to do that, we need to  investigate more closely the set $\{h_{i,n}\}_{i \in I_n}$, given by~\ref{lem::H_left_cosets}. Even if $h_{i,n}$ is uniquely determined by $k_{i,n}$, for every $i \in I_n$, we might still have two $h_{i,n}, h_{j,n}$, with $i \neq j \in I_n$, belonging to the same right coset of $H/(H\cap G_{\xi}^{0})$, thus $\Delta_{H}(h_{i,n})=\Delta_{H}(h_{j,n})$ by~\ref{rem::H_unimodular}. 

The evaluation of the set of all right cosets $[h_{i,n}] \in  H/(H\cap \G_{\xi}^{0})$ follows from~\ref{prop::evaluation_on_A}. Indeed, for simplicity set $g_n:=f_{2}^{-1}t_nf_1$. By~\ref{lem::KAK_decomposition}, one can write $g_n=k\gamma_nk'$, where $k,k' \in K$ and $\gamma_n \in A^{+}$ and there is a liberty to choose such $\gamma_n \in A^{+}$ and $k,k' \in K$.  We can choose $\gamma_n$ with $\proj_{(x,\xi]}(K g_nK)=\proj_{(x,\xi]}(\gamma_n)$. Fix such $\gamma_n, k,k'$ with $g_n=k\gamma_nk'$ and $\proj_{(x,\xi]}(Kg_nK)=\proj_{(x,\xi]}(\gamma_n)$.
\end{remark}

\begin{theorem}
\label{thm::vanishing_unimodular_case_1}

Let $F$ be primitive and let $\xi \in \bd \T$. Let $H$ be a closed, unimodular, subgroup of $\G_{\xi}$, containing hyperbolic elements and let $(\sigma, \mathcal{K})$ be a unitary representation of $H$. Then the induced unitary representation $(\pi_{\sigma}, \mathcal{H_{\sigma}})$ on $\G$ is $C_0$.
\end{theorem}

\begin{proof}
By~\ref{lem::hyp_H}, let $\gamma$ be a minimal hyperbolic element of $H$. Fix for what follows a vertex $x \in \Min(\gamma)$ and set $K:= \G_{x}$. By~\ref{lem::modular_fct_H}, $H_x= H_{y}= H_{[\xi_{-}, \xi]}$, for every $y \in \Min(\gamma)$, where $\xi_{-} \in \bd \T$ is the repelling endpoint of $\gamma$. By~\ref{lem::K-inv_borel_measure} and~\ref{rem::equiv_induced_unit_rep} we can consider, without loss of generality, the rho-function $\rho$ equals the constant function $\mathbf{1}$ on $\G$. In this particular case, the measure $\mu$ on $\G/H$ associated with the rho-function $\mathbf{1}$ on $\G$ is $\G$--invariant.  Apply~\ref{rem::the_strategy_first_step} and then~\ref{rem::strategy_second_step}, keeping all the notation there. By~\ref{lem::H_left_cosets} and~\ref{cor::evaluation_on_A_H_unimodular}, applied to $\gamma_n$, we have, for every $i \in I_{n}$, $ k_{i,n}^{-1}g_nk_{k_{i,n}}= k_{i,n}^{-1}k\gamma_nk'k_{k_{i,n}}=h_{i,n}= \gamma^{m_i}h_{0}$, with $\vert m_i\vert = \frac{\dist_{\T}(x,\gamma_n(x))}{\vert \gamma \vert}$ and $h_0 \in H\cap K$. Evaluate now the solutions for the equation 
\begin{equation}
\label{equ::equation_to_solv_unimod}
k_1 g_n k_2=k_1 k\gamma_nk'k_2 =h,
\end{equation} for a given right coset $ [h] \in \{[h_{i,n}] \; \vert \; i \in I_n\} \subset H/(H\cap \G_{\xi}^{0})$ and where $k_1k \in K$ and $k'k_2 \in K \setminus (K\cap H)$.  Note for any element $h \in H$ satisfying equation~(\ref{equ::equation_to_solv_unimod}) we have 
\begin{equation}
\label{equ::dist_for_h_unimod}
\dist_{\T}(x, h(x))=\dist_{\T}(x, g_n(x))= \dist_{\T}(x, \gamma_n(x))=\vert m \vert \dist_{\T}(x, \gamma(x)),
\end{equation} where $h= \gamma^{m}h_0$, with $h_0 \in H\cap K=H_x$. Apply again~\ref{cor::evaluation_on_A_H_unimodular}.  We obtain for a given right coset $ [h=\gamma^{m}] \in \{[h_{i,n}] \; \vert \; i \in I_n\} \subset H/(H\cap \G_{\xi}^{0})$ we have: (1) If $m>0$ then $k_1^{-1} \in k\G_{[x,x_h]}$, where $x_h \in [x, \xi]$ with $\dist_{\T}(x, x_h)= \dist_{\T}(x,\gamma_n(x))$; (2) If $m<0$ then $k_1^{-1} \in kk_3\G_{[x,\gamma^{m}(x)]}$, where $k_3 \in K$ with $[x, \gamma_n(x)]=k_3([x, \gamma^{m}(x)])$. 

To resume, for a fixed $n >0$ we have:
\begin{equation*}
\label{equ::final_sum_unimod}
\begin{split}
&\int\limits_{\sqcup_{i \in I_{n}}f_{2}k_{i,n}H} \Delta_{H}(h_{i,n})^{-1/2} d\mu(f_{2}k_{i,n}H)= \int\limits_{\sqcup_{i \in I_{n}}f_{2}k_{i,n}H} \mathbf{1} d\mu(f_{2}k_{i,n}H)\\
& \leq \int\limits_{kG_{[x,x_h]}H}  \mathbf{1} d\mu(f_{2}k\G_{[x,x_h]}H)  + \int\limits_{kk_3G_{[x,\gamma^{m}(x)]}H} \mathbf{1} d\mu(f_{2}kk_3\G_{[x,\gamma^{m}(x)]}H)\\
&=\mu(f_{2}k\G_{[x,x_h]}H) + \mu(f_{2}kk_3\G_{[x,\gamma^{m}(x)]}H).\\
\end{split}
\end{equation*}
As $t_n \xrightarrow[n \to\infty]{}  \infty$, we also have $g_n \to \infty$; thus by relation~(\ref{equ::dist_for_h_unimod}) $\dist_{\T}(x, x_h)=\dist_{\T}(x, \gamma_n(x))=\vert m \vert \dist_{\T}(x, \gamma(x)) \to \infty,$ when $n \to \infty$. By hypothesis, $\G_{\xi}, \G_{\xi_{-}}$ are closed, non-compact and proper subgroups of $\G$. By~\ref{lem::infinit_index} applied to $\G_{\xi}, \G_{\xi_{-}}$ we have $[K: \G_{[x, \xi]}] = \infty=[K: \G_{[x, \xi_{-}]}]$. Therefore $[K:G_{[x,\gamma^{m}(x)]}] \xrightarrow[m \to\infty]{} \infty$  and $[K:G_{[x,x_h]}] \to \infty $, when $n \to \infty$. By the $\G$--invariance of $\mu$ and because we have supposed  $\mu(KH)=1$ we claim: 
\begin{equation}
\label{equ::going_zero_hyperbolic}
\begin{split}
&\mu(f_{2}k\G_{[x,x_h]}H) + \mu(f_{2}kk_3\G_{[x,\gamma^{m}(x)]}H) = \mu(\G_{[x,x_h]}H) + \mu(\G_{[x,\gamma^{m}(x)]}H)\\
&= [K:G_{[x,\gamma^{m}(x)]}]^{-1} + [K:G_{[x,x_h]}]^{-1} \xrightarrow[t_n \to\infty]{}  0.
\end{split}
\end{equation}Indeed, we only need to prove if $K = \sqcup_{j}  \; k_j \G_{[x,y]}$, for some $y \in (\xi_{-}, \xi)$, then $KH=\sqcup_{j} \; k_j \G_{[x,y]}H$. Suppose this is not the case, then there exist $j_1 \neq j_2$ with $(k_{j_1} \G_{[x,y]}H )\cap (k_{j_2} \G_{[x,y]}H) \neq \emptyset$. So $k_{j_1}k=k_{j_2}k'h$, for some $k,k' \in \G_{[x,y]} \leq K$ and $h \in H$. Then $h \in K \cap H= H_{[\xi_-, \xi]} \subset \G_{[x,y]}$. Thus $k_{j_1} \G_{[x,y]}=k_{j_2} \G_{[x,y]}$, which is a contradiction. The claim follows. Relation~(\ref{equ::going_zero_hyperbolic}) is a contradiction of our initial assumption $\vert \left\langle \pi_{\sigma}(t_{n})\eta_1, \eta_2 \right\rangle \vert \nrightarrow 0$ and the theorem stands proven.
\end{proof}

\begin{theorem}
\label{thm::vanishing_nonunimodular_case}
Let $F$ be primitive and let $\xi \in \bd \T$. Let $H$ be a closed, non-unimodular, subgroup of $\G_{\xi}$, containing hyperbolic elements and let $(\sigma, \mathcal{K})$ be a unitary representation of $H$. Then the induced unitary representation $(\pi_{\sigma}, \mathcal{H_{\sigma}})$ on $\G$ is $C_0$.
\end{theorem}

\begin{proof}
By~\ref{lem::hyp_H}, let $\gamma$ be a minimal hyperbolic element of $H$. Fix for what follows a vertex $x \in \Min(\gamma)$ and set $K:= \G_{x}$. Apply~\ref{rem::the_strategy_first_step} and then~\ref{rem::strategy_second_step}, keeping all the notation there. By~\ref{lem::H_left_cosets} and~\ref{prop::evaluation_on_A}, applied to $\gamma_n$, we have, for every $i \in I_{n}$, $ k_{i,n}^{-1}g_nk_{k_{i,n}}= k_{i,n}^{-1}k\gamma_nk'k_{k_{i,n}}=h_{i,n}= \gamma^{m_i}h_{0}$, with $0\leq \vert m_i\vert \leq  \frac{\dist_{\T}(x,\gamma_n(x))}{\vert \gamma \vert}$ and $h_0 \in H\cap \G_{\xi}^{0}$. Evaluate now the solutions for the equation 
\begin{equation}
\label{equ::equation_to_solv}
k_1 g_n k_2=k_1 k\gamma_nk'k_2 =h,
\end{equation} for a given right coset $ [h] \in \{[h_{i,n}] \; \vert \; i \in I_n\} \subset H/(H\cap \G_{\xi}^{0})$ and where $k_1k \in K$ and $k'k_2 \in K \setminus (K\cap H)$.  Note for any element $h \in H$, satisfying equation~(\ref{equ::equation_to_solv}), we have 
\begin{equation}
\label{equ::dist_for_h}
\dist_{\T}(x, h(x))=\dist_{\T}(x, g_n(x))= \dist_{\T}(x, \gamma_n(x)).
\end{equation} 

Apply again~\ref{prop::evaluation_on_A}. For a given right coset $ [h=\gamma^{m}] \in \{[h_{i,n}] \; \vert \; i \in I_n\} \subset H/(H\cap \G_{\xi}^{0})$ we have: 1) If $m>0$ then $k_1^{-1} \in k\G_{[x,x_m]}$, where $x_m \in [x, \xi]$ is with $\dist_{\T}(x, x_m)=\frac{ \dist_{\T}(x,g_n(x))+m \vert \gamma \vert}{2}$; 2) If $m=0$ then $k_1^{-1} \in k\G_{[x,x_0]}$, where $x_0 \in [x, \xi]$ is with $\dist_{\T}(x, x_0)=\frac{ \dist_{\T}(x,g_n(x))}{2}$; 3) If $m<0$ then $k_1^{-1} \in k\G_{[x,x_m]}$, where $x_m \in [x, \xi]$ is with $\dist_{\T}(x, x_m)=\frac{ \dist_{\T}(x,g_n(x))- \vert m \vert \cdot \vert \gamma \vert}{2}$. To resume, for a fixed $n >0$ and for $N:=\frac{\dist_{\T}(x,g_n(x))}{\vert \gamma \vert}$ we have: $\int\limits_{\sqcup_{i \in I_{n}}f_{2}k_{i,n}H} \Delta_{H}(h_{i,n})^{-1/2} d\mu(f_{2}k_{i,n}H) \leq$

$ \leq \sum\limits_{m=-N}^{-1} \;  \int\limits_{kG_{[x,x_{m}]}H}  \Delta_{H}(\gamma^{m})^{-1/2} d\mu(f_{2}k\G_{[x,x_m]}H) + \mu(f_{2}k\G_{[x,x_{0}]}H) $

$+ \sum\limits_{m=1}^{N} \;  \int\limits_{kG_{[x,x_{m}]}H}  \Delta_{H}(\gamma^{m})^{-1/2} d\mu(f_{2}k\G_{[x,x_m]}H)= \mu(f_2 k \G_{[x,x_{0}]}H)+$

$+ \sum\limits_{m=-N}^{-1} \Delta_{H}(\gamma)^{-m/2}  \cdot \mu(f_2 k\G_{[x,x_{m}]}H)+ \sum\limits_{m=1}^{N}  \Delta_{H}(\gamma)^{-m/2} \cdot \mu(f_2 k \G_{[x,x_{m}]}H).$

Note by~\ref{rem::radon_nykodym} there is a constant $C_1>0$ depending only on $f_2k$, $K$ and the rho-function of $\mu$, with $\mu(f_2 k \G_{[x,x_{m}]}H) \Delta_{H}(\gamma)^{-m/2}\leq C_1 \mu(\G_{[x,x_{m}]} H) \Delta_{H}(\gamma)^{-m/2}$,  for every $m \in [-N,N]$.  Thus: $\int\limits_{\sqcup_{i \in I_{n}}f_{2}k_{i,n}H} \Delta_{H}(h_{i,n})^{-1/2} d\mu(f_{2}k_{i,n}H) \leq $

$ \leq C_1 \left(\sum\limits_{m=-N}^{-1} \Delta_{H}(\gamma)^{-m/2}  \cdot \mu(\G_{[x,x_{m}]}H)+ \mu( \G_{[x,x_{0}]}H)\right)+$
 
 $ + C_1 \left( \sum\limits_{m=1}^{N}  \Delta_{H}(\gamma)^{-m/2} \cdot \mu( \G_{[x,x_{m}]}H) \right)$
  
$=C_1 \left(\sum\limits_{m=-N}^{-1} \Delta_{H}(\gamma)^{-m/2}  \cdot [K:\G_{[x,x_m]}]^{-1}+ [K:\G_{[x,x_0]}]^{-1}\right)+$ 

$+C_1 \left( \sum\limits_{m=1}^{N}  \Delta_{H}(\gamma)^{-m/2} \cdot [K:\G_{[x,x_m]}]^{-1} \right).$

The last equality follows because $[K:\G_{[x,x_m]}]^{-1}=\mu(\G_{[x,x_{m}]}H)$, for every $m \in \{-N, N\}$. Indeed, we only need to prove if $K = \sqcup_{j}  \; k_j \G_{[x,y]}$, for some $y \in [x, \xi]$, then $KH=\sqcup_{j} \; k_j \G_{[x,y]}H$. Suppose this is not the case, then there exist $j_1 \neq j_2$ with $(k_{j_1} \G_{[x,y]}H )\cap (k_{j_2} \G_{[x,y]}H) \neq \emptyset$. So $k_{j_1}k=k_{j_2}k'h$, for some $k,k' \in \G_{[x,y]} \leq K$ and $h \in H$. Then $h \in K \cap H= H_{[x, \xi]} \subset \G_{[x,y]}$. Thus $k_{j_1} \G_{[x,y]}=k_{j_2} \G_{[x,y]}$, which is a contradiction.  Note as $t_n \xrightarrow[n \to\infty]{}  \infty$, we also have $g_n \to \infty$; thus by (\ref{equ::dist_for_h}) $\dist_{\T}(x, x_0)=\frac{\dist_{\T}(x, g_n(x))}{2} =\frac{\dist_{\T}(x, \gamma_n(x))}{2}  \xrightarrow[n \to\infty]{} \infty.$ By hypothesis, $\G_{\xi}$ is a closed, non-compact and proper subgroup of $\G$. By~\ref{lem::infinit_index} applied to $\G_{\xi}$ we have $[K: \G_{[x, \xi]}] = \infty$. Therefore we must have $[K:\G_{[x,\gamma^{l}(x)]}] \xrightarrow[l \to\infty]{} \infty$  and $[K:\G_{[x,x_0]}] \to \infty $, when $n \to \infty$. Apply~\ref{lem::convergence_zero} below and we will contradict our initial assumption $\vert \left\langle \pi_{\sigma}(t_{n})\eta_1, \eta_2 \right\rangle \vert \nrightarrow 0$ and the theorem stands proven.
\end{proof}

\begin{lemma}
\label{lem::convergence_zero}
Using the same notation as in the proof of~\ref{thm::vanishing_nonunimodular_case}
 we have:
\begin{list}{\arabic{qcounter})~}{\usecounter{qcounter}}
\item
\label{lem::convergence_zero_one} 
$\lim\limits_{N \to \infty} \sum\limits_{m=-N}^{-1} \Delta_{H}(\gamma)^{-m/2}  \cdot [K:\G_{[x,x_m]}]^{-1}=0$
\item
\label{lem::convergence_zero_two} 
$\lim\limits_{N \to \infty}\sum\limits_{m=1}^{N}  \Delta_{H}(\gamma)^{-m/2} \cdot [K:\G_{[x,x_m]}]^{-1}=0.$
\end{list}
\end{lemma}

\begin{proof}
Recall  $N:=\frac{\dist_{\T}(x,g_n(x))}{\vert \gamma \vert}$ and $\frac{\dist_{\T}(x, g_n(x))}{2} =\frac{\dist_{\T}(x, \gamma_n(x))}{2}  \xrightarrow[n \to\infty]{} \infty.$ Moreover, if $m>0$ then $\dist_{\T}(x, x_m)=\frac{ \dist_{\T}(x,g_n(x))+m \vert \gamma \vert}{2}$ and if $m<0$ then $\dist_{\T}(x, x_m)=\frac{ \dist_{\T}(x,g_n(x))- \vert m \vert \cdot \vert \gamma \vert}{2}$. By \ref{lem::modular_fct_H} and the hypothesis $H$ is non-unimodular we have $t:=\Delta_{H}(\gamma)=\frac{1}{[H_{[\gamma(x), \xi]}: H_{[x,\xi]}]} <1$. As $H \leq \G_{\xi} \leq \G$ we also have $[H_{[\gamma^m(x), \xi]}: H_{[x,\xi]}] \leq [\G_{[\gamma^m(x), \xi]}: \G_{[x,\xi]}]$, for every $m \geq 0$. By \ref{lem::hyp_element}, we have $[K: \G_{[x, \gamma^{m}(x)]}]= \frac{[\G_{[\gamma^{m}(x), \xi]}: \G_{[x, \xi]}] \cdot d}{k_1}$, for every $m > 0$, where  $k_1$ is the number of orbits of the edge $e_-$ in $\{1, \cdots,d\}$ under the stabilizer subgroup $F_{e_+} \leq F$. Let $ 0 \leq l(m):=\lfloor m/2 + N/2 \rfloor$ the integer value of $\frac{ \sign(m) \vert \gamma^{m}(x)\vert +\vert \gamma_n(x) \vert}{2 \cdot \vert \gamma \vert}= m/2 + N/2$, for every $m \geq -N$. Thus, for every $m \geq -N$: $[K: \G_{[x, \gamma^{l(m)}(x)}] \leq$
\begin{equation}
\label{equ::bigger}
 \leq [K:\G_{[x,x_m]}]=[K:\G_{[x,\frac{\sign(m) \vert \gamma^{m}(x)\vert +\vert \gamma_n(x) \vert}{2}]}] \leq [K: \G_{[x, \gamma^{l(m)+1}(x)}].
\end{equation}

Let us prove the assertion~\ref{lem::convergence_zero_two}). By \ref{lem::modular_fct_H}, $t^{-m/2}=[H_{[\gamma^m(x), \xi]}: H_{[x,\xi]}]^{1/2}$, for every $m \geq 0$. By \ref{lem::index_fct_H} applied to $\G_{\xi}$, for every $0 \leq m$ we have
\begin{equation}
\label{equ::bigger_square}
[\G_{[\gamma^m(x), \xi]}: \G_{[x,\xi]}] \leq [\G_{[\gamma^{\lfloor m/2 \rfloor +1}(x), \xi]}: \G_{[x,\xi]}]^{2}.
\end{equation}

\medskip
Using (\ref{equ::bigger}) and~(\ref{equ::bigger_square}), the assertion~\ref{lem::convergence_zero_two}) becomes: $\sum\limits_{m=1}^{N}  \Delta_{H}(\gamma)^{-m/2} \cdot [K:\G_{[x,x_m]}]^{-1}$

$= \sum\limits_{m=1}^{N}  [H_{[\gamma^m(x), \xi]}: H_{[x,\xi]}]^{1/2} \cdot [K:\G_{[x,x_m]}]^{-1} \leq \sum\limits_{m=1}^{N} \frac{[\G_{[\gamma^m(x), \xi]}: \G_{[x,\xi]}]^{1/2}}{[K:\G_{[x,x_m]}]} \leq $

$ \leq \sum\limits_{m=1}^{N} \frac{[\G_{[\gamma^{\lfloor m/2 \rfloor +1}(x), \xi]}: \G_{[x,\xi]}]}{[K:\G_{[x,\gamma^{l(m)}(x)}]}=\frac{k_1}{d} \sum\limits_{m=1}^{N} \frac{[\G_{[\gamma^{\lfloor m/2 \rfloor +1}(x), \xi]}: \G_{[x,\xi]}]}{[\G_{[\gamma^{l(m)}(x), \xi]}: \G_{[x,\xi]}]} \leq $

$\leq \frac{k_1}{d} \cdot \frac{N}{[\G_{[\gamma^{\lfloor N/2 \rfloor-1}(x), \xi]}: \G_{[x,\xi]}]}= \frac{k_1}{d} \cdot N \cdot \Delta_{\G_{\xi}}(\gamma)^{\lfloor N/2 \rfloor-1} \xrightarrow[N \to\infty]{} 0.$

\medskip
Using (\ref{equ::bigger}), the assertion~\ref{lem::convergence_zero_one}) becomes: $\lim\limits_{N \to \infty} \sum\limits_{m=-N}^{-1} \Delta_{H}(\gamma)^{-m/2}  \cdot [K:\G_{[x,x_m]}]^{-1}$

$=\lim\limits_{N \to \infty} \sum\limits_{m=-N}^{-1} t^{\vert m \vert /2}  \cdot [K:\G_{[x,x_m]}]^{-1}  \leq \lim\limits_{N \to \infty} \sum\limits_{m=-N}^{-1} t^{\vert m \vert /2}  \cdot [K:\G_{[x,\gamma^{l(m)}(x)]}]^{-1}$

$= \lim\limits_{N \to \infty}\left( t^{N/2} \cdot 1+ \sum\limits_{m=-N+1}^{-1} t^{\vert m \vert /2}  \cdot [K:\G_{[x,\gamma^{l(m)}(x)]}]^{-1}\right) \leq$

$ \leq \lim\limits_{N \to \infty} \sum\limits_{m=-N+1}^{(-N+1)/2} t^{\vert m \vert /2}  \cdot [K:\G_{[x,\gamma^{l(m)}(x)]}]^{-1}+$

$+ \lim\limits_{N \to \infty} \sum\limits_{m=(-N+1)/2}^{-1} t^{\vert m \vert /2}  \cdot [K:\G_{[x,\gamma^{l(m)}(x)]}]^{-1} \leq \lim\limits_{N \to \infty} \sum\limits_{m=-N+1}^{(-N+1)/2} t^{\vert m \vert /2}+$

$ + \lim\limits_{N \to \infty} \sum\limits_{m=(-N+1)/2}^{-1} t^{\vert m \vert /2}  \cdot [K:\G_{[x,\gamma^{l(m)}(x)]}]^{-1} \leq \lim\limits_{N \to \infty} \frac{N+1}{2} \cdot t^{(N-1)/2}+$
 
$+ \lim\limits_{N \to \infty}  \; [K:\G_{[x,\gamma^{l((-N+1)/2)}(x)]}]^{-1} \cdot \sum\limits_{m=(-N+1)/2}^{0} t^{\vert m \vert /2} \; \; .$

As $t<1$, one has $\lim\limits_{N \to \infty} \frac{N+1}{2} \cdot t^{(N-1)/2} =0$. 

Moreover, $\lim\limits_{N \to \infty} [K:\G_{[x,\gamma^{l((-N+1)/2)}(x)]}]^{-1}=0$ and $\sum\limits_{m=(-N+1)/2}^{0} t^{\vert m \vert /2}=\frac{1}{1-t^{1/2}}$. Thus $\lim\limits_{N \to \infty}  \; [K:\G_{[x,\gamma^{l((-N+1)/2)}(x)]}]^{-1} \cdot \sum\limits_{m=(-N+1)/2}^{0} t^{\vert m \vert /2}=0$. 
\end{proof}

\section{The main Theorem}
\ref{cor::vanishing_unimod_case} and~\ref{thm::vanishing_unimodular_case} and~\ref{thm::vanishing_nonunimodular_case} give us the aimed result of this article:

\begin{theorem}
\label{thm::vanishing_general_case}
Let $F$ be primitive and let $\xi \in \bd \T$. Let $H$ be a closed subgroup of $\G_{\xi}$ and let $(\sigma, \mathcal{K})$ be a unitary representation of $H$. Then the induced unitary representation $(\pi_{\sigma}, \mathcal{H_{\sigma}})$ on $\G$ is $C_0$.
\end{theorem}

\begin{proof}
It remains to consider the case when $H$ is a compact subgroup of $\G_{\xi}$. This is a particular case of the well-known general fact that all unitary representations of a locally compact subgroup that are induced from compact subgroups are $C_0$. For the idea of the proof the reader can consult the book of Bekka--de la Harpe--Valette~\cite[Proposition~C.4.6]{BHV}.
\end{proof}

\subsection*{Acknowledgements} We would like to thank Pierre-Emmanuel Caprace for addressing the question if parabolically induced unitary representations of $\G$, with $F$ being primitive, are $C_0$. We thank Pierre-Emmanuel Caprace and Stefaan Vaes for pointing out a gap in an earlier version of this paper and Alain Valette for further discussions. The comments of the anonymous referee were highly appreciated. We would like to thank him/her for carefully reading this paper.


\end{document}